\pdfoutput=1
\documentclass[11pt]{article}

\usepackage[utf8]{inputenc}
\usepackage{microtype}
\usepackage{graphicx}
\usepackage[svgnames]{xcolor}
\usepackage[unicode,bookmarksnumbered,bookmarksopen,psdextra,colorlinks=true,allcolors=DarkBlue]{hyperref}
\usepackage{bookmark}
\usepackage{tocbibind}
\usepackage{array}
\usepackage{fullpage}
\usepackage{mymacros}

\newcommand*\Hom{\mathrm{Hom}}
\newcommand*\op{\mathsf{op}}
\newcommand*\dist{d}
\mathlig{<<}{\ll}
\mathlig{>>}{\gg}

\begin{document}

\title{A Gelfand duality for continuous lattices}
\author{Ruiyuan Chen}
\date{}
\maketitle

\begin{abstract}
We prove that the category of continuous lattices and meet- and directed join-preserving maps is dually equivalent, via the hom functor to $[0,1]$, to the category of complete Archimedean meet-semilattices equipped with a finite meet-preserving action of the monoid of continuous monotone maps of $[0,1]$ fixing $1$.
We also prove an analogous duality for completely distributive lattices.
Moreover, we prove that these are essentially the only well-behaved ``sound classes of joins $\Phi$, dual to a class of meets'' for which ``$\Phi$-continuous lattice'' and ``$\Phi$-algebraic lattice'' are different notions, thus for which a $2$-valued duality does not suffice.
\let\thefootnote=\relax
\end{abstract}

\section{Introduction}
\label{sec:intro}

The classical Gelfand duality asserts that a compact Hausdorff space $X$ may be recovered from its ring of continuous functions $C(X)$, and moreover such rings are up to isomorphism precisely the commutative $C^*$-algebras.
From a categorical perspective, $C(X)$ is best regarded as having ``underlying set'' given by its (positive) unit ball, i.e., consisting of continuous $\#I := [0,1]$-valued functions, so that Gelfand duality falls under the umbrella of Stone-type dualities induced by two ``commuting'' structures on $\#I$; see \cite[VI~\S4]{Jstone}.
Namely, $\#I$ is equipped with its usual compact Hausdorff topology, and also with all operations $\#I^\kappa -> \#I$ ``commuting'' with the topology, i.e., which are continuous.
Thus, for another object in either category, the hom functor into $\#I$ yields a dual in the other category, and this gives a dual adjunction, which Gelfand duality asserts is an equivalence.
An explicit axiomatization of the dual operations on the $\#I$-valued $C(X)$ was recently given in \cite{MRgelfand}; see there for a detailed history of $\#I$-valued Gelfand duality.
In \cite{HNNnachbin}, \cite{Anachbin}, $\#I$-valued Gelfand duality was further extended to compact partially ordered spaces (\emph{a la} Nachbin).

In this note, we prove analogous Gelfand-type dualities for compact pospaces equipped with lattice operations.
Recall that a \defn{continuous lattice} is a compact topological meet-semilattice obeying a ``local convexity under meets'' condition, that each point has a neighborhood basis of subsemilattices.
Equivalently, they can be defined purely order-theoretically as posets with arbitrary meets distributing over directed joins.
An analog of Urysohn's lemma, sometimes known as the Urysohn--Lawson lemma, states that every continuous lattice $X$ admits enough morphisms to $\#I$, i.e., the canonical evaluation map $X -> \#I^{\Hom(X, \#I)}$ is an embedding; see \cite[IV-3.3]{GHKLMS}, \cite[VII~3.2]{Jstone}.
It is thus natural to ask whether, by equipping $\Hom(X, \#I)$ with suitable structure commuting with the continuous lattice structure on $\#I$, we may recover $X$ as the double dual.

Let $\^{\#U}$ denote the monoid of continuous monotone maps $\#I -> \#I$ fixing $1$, i.e., all unary operations on $\#I$ commuting with the continuous lattice structure.
Note that finite meets do as well.
By a \defn{$\^{\#U}$-module}, we mean a unital meet-semilattice equipped with an action of $\^{\#U}$ preserving finite meets in both variables.
In every $\^{\#U}$-module $A$, we have a canonical pseudoquasimetric
\begin{align*}
\rho(a, b) := \bigwedge \{r \in \#I \mid a \le b \dotplus r\}
\end{align*}
where $b \dotplus r$ denotes the result of the action on $b$ of the truncated addition $(-) \dotplus r \in \^{\#U}$.
We say $A$ is \defn{Archimedean} if $\rho(a, b) = 0 \implies a \le b$, and \defn{complete} if $A$ is Archimedean and complete with respect to the induced metric $\dist(a, b) := \rho(a, b) \vee \rho(b, a)$.
We prove

\begin{theorem}[\cref{thm:ctslat-dual-umod}]
\label{thm:intro-ctslat-dual-umod}
Hom into $\#I$ yields a dual equivalence of categories between continuous lattices and complete $\^{\#U}$-modules.
\end{theorem}

There is a generalization of continuous lattice theory, with the role of directed joins replaced by an arbitrary ``class of joins $\Phi$'' obeying suitable axioms; see \cite{WWTzpos}, \cite{BEzcts}, \cite{Xctslat}, as well as \cite{AKsat}, \cite{ABLRdacc}, \cite{KScolim} for a further extension in enriched category theory.
Other than $\Phi =$ ``directed joins'', the most well-known case is $\Phi =$ ``all joins'', for which $\Phi$-continuous lattices are completely distributive lattices.
As for continuous lattices, there is a Urysohn-type lemma, stating that all completely distributive lattices admit enough morphisms to $\#I$; see \cite[IV-3.31--32]{GHKLMS}, \cite[1.10--14]{Jstone}.
We likewise boost this to a Gelfand-type duality as follows.

Let $\#U \subseteq \^{\#U}$ denote the monoid of complete lattice morphisms, i.e., monotone surjections.
A \defn{$\#U$-poset} is a poset with a monotone action of $\#U$.
There is a canonical way of defining a pseudoquasimetric on a $\#U$-poset, agreeing with the above definition in $\^{\#U}$-modules; see \cref{def:upos-met}.
A $\#U$-poset $A$ is \defn{stackable} if, intuitively speaking, an element $a \in A$ may be specified via its ``restrictions to sublevel and superlevel sets $a^{-1}([0,r]), a^{-1}([r,1])$'' for any $0 < r < 1$; see \cref{def:upos-stack}.

\begin{theorem}[\cref{thm:cdlat-dual}]
\label{thm:intro-cdlat-dual}
Hom into $\#I$ yields a dual equivalence of categories between completely distributive lattices and complete stackable $\#U$-posets.
\end{theorem}

In fact, we prove a single result underlying \cref{thm:intro-ctslat-dual-umod,thm:intro-cdlat-dual}, for a ``class of joins $\Phi$ dual to a class of meets $\Psi^\op$'', more precisely for a \emph{sound} class of joins in the sense of \cite{ABLRdacc}, \cite{KScolim}; see \cref{sec:comm}.
This general result, \cref{thm:cts-dual}, says that $\Phi$-continuous lattices are dual to complete stackable $\#U$-$\Psi^\op$-inflattices, \emph{provided that not all $\Phi$-continuous lattices are $\Phi$-algebraic}, i.e., already admit enough morphisms into $2$.
This is a reasonable restriction, since for these other $\Phi$, we instead have a simple $2$-valued duality generalizing the classical Hofmann--Mislove--Stralka duality \cite{HMSlat} between algebraic lattices and meet-semilattices (see \cref{thm:hms}).

Part of the reason we work with general $\Phi$ is to hint at the possibility of generalizing to quantale-enriched posets, or even to enriched categories, which we plan to pursue in future work.
However, in the original context of mere posets, it turns out that essentially the only $\Phi$ are the classical ones:

\begin{theorem}[\cref{thm:sound-omega}]
There are precisely 4 sound classes of joins $\Phi$ for which not every $\Phi$-continuous lattice is $\Phi$-algebraic: ``directed joins'', ``all joins'', and the minor variations including/excluding empty joins.
\end{theorem}

\paragraph*{Acknowledgments}

I would like to thank the anonymous referee for numerous helpful comments and suggestions that improved the presentation of the paper.
Research partially supported by NSF grant DMS-2224709.

\section{$\Phi$-continuous lattices}

We assume familiarity with basic category theory.
For a category $\!C$, $\!C(X,Y)$ will denote the hom-set of morphisms from $X$ to $Y$, while $\!C^\op$ will denote the opposite category; this includes opposite posets.
We let $\!{Pos}$ denote the category of posets, $\!{Sup}$ denote the category of suplattices (i.e., complete lattices with join-preserving maps as morphisms), $\!{Inf}$ denote the category of inflattices, and $\!{CLat} = \!{Sup} \cap \!{Inf}$ denote the category of complete lattices.
These are all locally ordered categories: each hom-set is partially ordered pointwise, and composition is monotone on both sides.
For $f : X -> Y \in \!{Pos}$ left adjoint to $g : Y -> X$, we will write $f = g^+$ and $g = f^\times$.
We will frequently use the ``mate calculus'': for monotone $h, k$, we have $h \circ g \le k \iff h \le k \circ f$.

For a poset $X$, we let $\@L(X)$ denote the poset of lower sets $\phi \subseteq X$, ordered via $\subseteq$.
Then $\@L : \!{Pos} -> \!{Pos}$ is the free suplattice monad, where the monad structure consists of:
\begin{itemize}
\item  unit $\down = \down_X : X -> \@L(X)$, where $\down x = \{y \in X \mid y \le x\}$ is the principal ideal below $x$;
\item  multiplication ${\bigcup} : \@L(\@L(X)) -> \@L(X)$;
\item  $f : X -> Y \in \!{Pos}$ inducing $f_* = \@L(f) : \@L(X) -> \@L(Y) \in \!{Sup}$, where $f_*(\phi) = \bigcup_{x \in \phi} \down f(x)$.
\end{itemize}

We now review the theory of ``relative'' suplattices for a ``class of joins'' $\Phi$.
This is a special case of the theory of ``classes of colimits'' in enriched category theory \cite{AKsat}, \cite{ABLRdacc}, \cite{KScolim}, and has also been well-studied in the order theory literature as ``$Z$-completeness'' \cite{WWTzpos}, \cite{BEzcts}.
We will use notation and terminology based on that from enriched categories.

\begin{definition}
\label{def:phi}
A \defn{join doctrine} is a class $\Phi$ of posets $\phi$, thought of as indexing posets for certain joins $\bigvee_{x \in \phi} f(x)$ of monotone $f : \phi -> Y$.
We require $\Phi$ to obey the following ``saturation'' conditions:
\begin{enumerate}[label=(\roman*)]
\item \label{def:phi:unit}
The singleton poset $\*1$ is in $\Phi$.
\item \label{def:phi:mult}
If $\phi$ is a poset which is a union $\bigcup \Psi$ of a set $\Psi \subseteq \Phi$ of subposets $\psi \subseteq \phi$ which are in $\Phi$, and also $\Psi$ (as a poset under $\subseteq$) is in $\Phi$, then $\phi \in \Phi$.
\item \label{def:phi:funct}
If $f : \phi -> \psi$ is a monotone map with cofinal image, and $\phi \in \Phi$, then $\psi \in \Phi$.
\item \label{def:phi:proper}
If $\phi \subseteq \psi$ is a cofinal subposet, and $\psi \in \Phi$, then $\phi \in \Phi$.
\end{enumerate}
A \defn{$\Phi$-join} in a poset $X$ is a join of a subset $\phi \subseteq X$ such that $\phi \in \Phi$.
A \defn{$\Phi$-suplattice} is a poset with all $\Phi$-joins; we denote the category of all such (and monotone $\Phi$-join-preserving maps) by $\Phi\!{Sup}$.
A \defn{$\Phi$-ideal} in a $\Phi$-suplattice is a lower sub-$\Phi$-suplattice.
The \defn{free $\Phi$-suplattice} generated by a poset $X$ is the subset $\Phi(X) \subseteq \@L(X)$ of all lower subsets of $X$ in $\Phi$.
Note that for a poset $\phi$, we have $\phi \in \Phi \iff \phi \in \Phi(\phi)$; we thereby identify the class of posets $\Phi$ with the submonad $\Phi \subseteq \@L$.
\end{definition}

\begin{example}
\leavevmode
\begin{itemize}
\item
The ``class of directed joins'' is given by the join doctrine $\Phi :=$ all directed posets, for which a $\Phi$-suplattice is a directed-complete poset (DCPO), a $\Phi$-ideal is a Scott-closed subset, and $\Phi(X)$ is the ideal completion of $X$ (note: \emph{not} ``$\Phi$-ideal completion'').
\item
The ``class of finite joins'' is given by $\Phi :=$ all posets with finite cofinality.
\item
The ``class of all joins'' is given by $\Phi :=$ all posets.
\item
The least join doctrine, of ``trivial joins'', is given by $\Phi :=$ posets with a greatest element.
\end{itemize}
\end{example}

\begin{remark}
\label{rmk:phi-proper}
In \cite{AKsat} and \cite{KScolim}, a more general notion of ``class of colimits'' is considered, consisting in the posets case of an arbitrary submonad $\Phi \subseteq \@L$, i.e., an assignment to each poset $X$ of a set of lower sets $\Phi(X) \subseteq \@L(X)$ closed under the monad operations on $\@L$.

The precise connection with our definition of ``join doctrine'' as a class of posets is as follows.
Each join doctrine $\Phi$ induces a free $\Phi$-suplattice submonad as above; this yields an order-embedding
\begin{align*}
\{\text{join doctrines}\} `--> \{\text{submonads of } \@L\},
\end{align*}
whose image consists of those submonads $\Phi \subseteq \@L$ obeying the additional ``saturation'' condition
\begin{enumerate}
\refitem{($*$)}
\label{rmk:phi-proper:proper}
for each order-embedding between posets $f : X `-> Y$, we have $\Phi(X) = f_*^{-1}(\Phi(Y))$.
\end{enumerate}
This condition is implied by condition \cref{def:phi:proper} in \cref{def:phi} of join doctrine, and conversely, ensures that $\{\phi \in \!{Pos} \mid \phi \in \Phi(\phi)\}$ is a join doctrine inducing the submonad $\Phi$.

An example of a submonad not obeying \cref{rmk:phi-proper:proper} is $\Phi(X) := \{\phi \in \@L(X) \mid \phi \text{ has an upper bound in } X\}$, which yields the ``class of bounded joins''.
However, \cref{rmk:phi-proper:proper} is automatic for the $\Phi$ suitable for our duality purposes, which is why we use the simpler definition of ``join doctrine''; see \cref{rmk:comm-proper}.
\end{remark}

\begin{definition}
\label{def:phi-way}
Let $\Phi$ be a join doctrine, $X$ be a $\Phi$-suplattice.
We define, for $x, y \in X$,
\begin{gather*}
\begin{aligned}
\waydown = \waydown^\Phi_X : X &--> \@L(X) \\
x &|--> \bigcap \{\phi \in \Phi(X) \mid x \le \bigvee \phi\},
\end{aligned} \\
x << y  \coloniff  x <<^\Phi y  \coloniff  x \in \waydown y.
\end{gather*}
We call $x \in X$ \defn{$\Phi$-compact} (\emph{$\Phi$-atomic} in \cite{KScolim}) if $x <<^\Phi x$, i.e., whenever $\bigvee_i y_i$ is a $\Phi$-join $\ge x$, then some $y_i \ge x$, i.e., the indicator function of $\up x : X -> 2$ preserves $\Phi$-joins.
Denote these by
\begin{equation*}
X_\Phi := \{x \in X \mid x <<^\Phi x\}.
\end{equation*}
We call $X$ \defn{$\Phi$-algebraic} if it is generated under $\Phi$-joins by $X_\Phi \subseteq X$.
In that case, it is easy to see that in fact, for each $x \in X$ the set $X_\Phi \cap \down x$ belongs to $\Phi(X_\Phi)$ and has join $x$; and this yields an order-isomorphism $X \cong \Phi(X_\Phi)$.
Conversely, for any poset $Y$, we easily have that $\Phi(Y)$ is $\Phi$-algebraic, with $\Phi(Y)_\Phi = \{\text{principal ideals}\} \cong Y$.
\end{definition}

\begin{proposition}
\label{thm:cts}
Let $\Phi$ be a join doctrine, $X$ be a $\Phi$-suplattice.
The following are equivalent:
\begin{enumerate}[label=(\roman*)]
\item \label{thm:cts:join}
For each $x \in X$, there is a $\phi \in \Phi(X)$ such that $\phi \subseteq \waydown x$ and $x \le \bigvee \phi$, whence in fact $\phi = \waydown x$.
\item \label{thm:cts:way}
${\bigvee} : \Phi(X) -> X$ has a left adjoint, namely $\waydown$.
\end{enumerate}
If $X$ is a complete lattice, these are further equivalent to:
\begin{enumerate}[resume*]
\item \label{thm:cts:meet}
${\bigvee} : \Phi(X) -> X$ preserves meets.
\item \label{thm:cts:dist}
Arbitrary meets distribute over $\Phi$-joins: if $\bigvee_{j \in J_i} x_{i,j}$ is a $\Phi$-join for each $i \in I$, then
\begin{align*}
\bigwedge_{i \in I} \bigvee_{j \in J_i} x_{i,j} = \bigvee_{(j_i)_i \in \prod_i J_i} \bigwedge_{i \in I} x_{i,j_i}.
\end{align*}
\end{enumerate}
All of these hold if $X$ is algebraic, with $\waydown = \down_* : \Phi(X_\Phi) -> \Phi(\Phi(X_\Phi))$, i.e.,
\begin{equation*}
x << y \iff \exists z \in X_\Phi\, (x \le z \le y).
\end{equation*}
\end{proposition}
If \cref{thm:cts:join}, \cref{thm:cts:way} hold for a $\Phi$-suplattice $X$, we call $X$ \defn{$\Phi$-continuous}.
If furthermore $X$ is a complete lattice, we call $X$ a \defn{$\Phi$-continuous lattice}, or a \defn{$\Phi$-algebraic lattice} if $X$ is algebraic.
\begin{proof}
\cref{thm:cts:join}$\iff$\cref{thm:cts:way} since it is easily seen that $\phi$ in \cref{thm:cts:join} must be $\waydown x$.

\cref{thm:cts:way}$\iff$\cref{thm:cts:meet} by the adjoint functor theorem.

\cref{thm:cts:meet}$\iff$\cref{thm:cts:dist} because the latter says
$\bigwedge_{i \in I} \bigvee \bigcup_{j \in J_i} \down x_{i,j_i} = \bigvee \bigcap_{i \in I} \bigcup_{j \in J_i} \down x_{i,j_i}$.
\end{proof}

\begin{proposition}
\label{thm:way-props}
In every $\Phi$-suplattice,
\begin{enumerate}[label=(\alph*)]
\item
$\waydown x \subseteq \down x$, i.e., $y << x \implies y \le x$.
\item
$x' \le x << y \le y' \implies x' << y'$.
\end{enumerate}
In a $\Phi$-continuous $\Phi$-suplattice,
\begin{enumerate}[resume*]
\item \label{thm:way-props:interp}
(interpolation) $\waydown = \bigcup \waydown_* \waydown$, i.e.,
$\waydown x = \bigcup_{y << x} \waydown y$, i.e.,
\begin{equation*}
z << x  \iff  \exists y\, (z << y << x).
\end{equation*}
\end{enumerate}
\end{proposition}
\begin{proof}
The first two are obvious.
For interpolation:
since $X$ is an algebra of the monad $\Phi$, we have
$\bigvee \bigcup = \bigvee \bigvee_* : \Phi(\Phi(X)) -> X$;
taking left adjoints yields
$\down_* \waydown = \waydown_* \waydown$;
now take $\bigcup$.
\end{proof}

A \defn{morphism of $\Phi$-continuous lattices} is a meet-preserving, $\Phi$-join-preserving map between $\Phi$-continuous lattices.
Let $\Phi\!{CtsLat}$ denote the category of $\Phi$-continuous lattices and morphisms, and $\Phi\!{AlgLat} \subseteq \Phi\!{CtsLat}$ denote the full subcategory of $\Phi$-algebraic lattices.

\begin{proposition}
\label{thm:cts-ladj}
Let $f : X -> Y$ be a right adjoint between $\Phi$-continuous $\Phi$-suplattices, with left adjoint $f^+ : Y -> X$.
Then $f$ preserves $\Phi$-joins iff $f^+$ preserves $<<$.
Thus
\begin{align*}
\Phi\!{CtsLat}(X, Y)^\op &\cong {<<^\Phi}\!{Sup}(Y, X) := \{f^+ : Y -> X \mid f^+ \text{ preserves $<<, \bigvee$}\} \\
f &|-> f^+.
\end{align*}
\end{proposition}
\begin{proof}
$f \bigvee = \bigvee f_* : \Phi(X) -> Y$ iff, taking left adjoints,
$\waydown f^+ = (f^+)_* \waydown : Y -> \Phi(X)$.
\end{proof}

\begin{proposition}
\label{thm:phi-cts}
Let $\Phi$ be a join doctrine.
The following are equivalent:
\begin{enumerate}[label=(\roman*)]
\item \label{thm:phi-cts:inf}
For every complete lattice $X$, $\Phi(X) \subseteq \@L(X)$ is closed under meets.
\item \label{thm:phi-cts:low}
For every poset $X$, $\Phi(\@L(X)) \subseteq \@L(\@L(X))$ is closed under meets.
\item \label{thm:phi-cts:cts}
For every poset $X$, $\@L(X)$ is $\Phi$-continuous.
\end{enumerate}
\end{proposition}
If these conditions hold, we call $\Phi$ a \defn{continuous} join doctrine.
\begin{proof}
\cref{thm:phi-cts:inf}$\implies$\cref{thm:phi-cts:low} is obvious.

\cref{thm:phi-cts:low}$\implies$\cref{thm:phi-cts:cts} since $\bigcup : \Phi(\@L(X)) -> \@L(X)$ is the composite of the inclusion $\Phi(\@L(X)) `-> \@L(\@L(X))$ and $\bigcup : \@L(\@L(X)) -> \@L(X)$, which both preserve meets, i.e., have left adjoints.

\cref{thm:phi-cts:cts}$\implies$\cref{thm:phi-cts:inf} since the composite $\@L(X) --->{\waydown_{\@L(X)}} \Phi(\@L(X)) --->{\bigvee_*} \Phi(X)$ yields the $\Phi(X)$-closure of each lower set $\psi$: we have
$1_{\@L(X)} \le \bigvee_* \waydown_{\@L(X)}$ because ${\bigcup} \le {\bigvee_*} : \Phi(\@L(X)) -> \Phi(X) \subseteq \@L(X)$, while $\bigvee_* \waydown_{\@L(X)}$ restricted to $\Phi(X) \subseteq \@L(X)$ becomes $\bigvee_* \down_* = 1_{\Phi(X)}$.
\end{proof}

The following are the two main examples of continuous join doctrines:

\begin{example}
\label{ex:cts-dir}
If $\Phi$ is the ``class of directed joins'', i.e., the class of all directed posets, so that $\Phi(X)$ for $X \in \!{Pos}$ is the ideal completion of $X$, then $<<$ is the classical way-below relation, and $\Phi$-continuity and $\Phi$-algebraicity become classical continuity and algebraicity for DCPOs.

Similarly, for any infinite regular cardinal $\kappa$, one can consider $\kappa$-directed joins.
But it turns out that for uncountable $\kappa$, continuity and algebraicity coincide; see \cref{thm:cts-alg}.
\end{example}

\begin{example}
\label{ex:cts-cd}
If $\Phi$ is the ``class of all joins'', i.e., the class of all posets, so that $\Phi(X) = \@L(X)$, then a $\Phi$-continuous lattice is a completely distributive lattice, and $<<$ is the ``way-way-below'' relation sometimes denoted $\lll$; see e.g., \cite[IV-3.31]{GHKLMS}.
\end{example}

Minor variations are to include/exclude empty joins, which only affects $\Phi$-compactness of $\bot$.

\begin{example}[the unit interval]
\label{ex:cts-r}
For any join doctrine $\Phi$, $\#I := [0,1]$ is a $\Phi$-continuous lattice.
Indeed, $<<$ contains $<$, since any $\phi \in \@L(\#I)$ with $r \le \bigvee \phi$ must clearly contain $[0,r)$; thus $r = \bigvee \waydown r$.
\end{example}

We now completely characterize the $<<^\Phi$ relation on $\#I$, by determining which $r \in \#I$ are $\Phi$-compact.

\begin{proposition}
\label{thm:r-way}
Let $\Phi$ be a join doctrine.
\begin{enumerate}[label=(\alph*)]
\item \label{thm:r-way:0}
For every $\Phi$-suplattice $X$, $\bot \in X$ is $\Phi$-compact iff $\emptyset \not\in \Phi$.
In particular, this holds for $0 \in \#I$.
\item \label{thm:r-way:omega}
If $\omega \in \Phi$ (where $\omega$ has the usual linear order), then no $r > 0$ is $\Phi$-compact in $\#I$.
Otherwise:
\begin{enumerate}[label=(\roman*)]
\item \label{thm:r-way:nomega:seq}
For every $\phi \in \Phi$ and $x_0, x_1, \dotsc \in \phi$, there are $i_0 < i_1 < \dotsb$ such that $x_{i_0}, x_{i_1}, \dotsc$ have an upper bound in $\phi$.
In particular, every $x_0 \le x_1 \le \dotsb \in \phi$ has an upper bound.
\item \label{thm:r-way:nomega:alg}
Every $\Phi$-continuous $\Phi$-suplattice $X$ which also has countable increasing joins is $\Phi$-algebraic, with the join of any $x_0 << x_1 << \dotsb \in X$ being $\Phi$-compact.
In particular, every $r > 0$ is $\Phi$-compact in $\#I$.
\end{enumerate}
\end{enumerate}
\end{proposition}
\begin{proof}
\cref{thm:r-way:0} is clear from the definition of $\Phi$-compact.

\cref{thm:r-way:omega}
If $\omega \in \Phi$, then no $r > 0$ is $\Phi$-compact, since $r$ is the join of a sequence in $[0,r)$.
Now suppose $\omega \not\in \Phi$.
Then for $\phi \in \Phi$ and $x_0, x_1, \dotsc \in \phi$, if no infinite subfamily has an upper bound, then we have a monotone map $\phi -> \omega$ taking $\phi \setminus \bigcup_n \up x_n$ to $0$ and each $\up x_n \setminus \bigcup_{m > n} \up x_m$ to $n+1$; since $\omega \not\in \Phi$, this map must have finite image, whence there are $i_0 < i_1 < \dotsb$ with $x_{i_0} \ge x_{i_1} \ge \dotsb$, a contradiction, which proves \cref{thm:r-way:nomega:seq}.
It follows that for a $\Phi$-continuous $\Phi$-suplattice $X$ with countable increasing joins, every $\waydown x \in \Phi(X)$ is closed under countable increasing joins.
In particular, for $x_0 << x_1 << \dotsb \in X$, $x := \bigvee_n x_n$ has $x_n << x$ for each $n$, whence $x << x$.
Now for any $y \in X$ and $x_0 << y$, by interpolation (\cref{thm:way-props}\cref{thm:way-props:interp}) we may find $x_0 << x_1 << \dotsb << y$, whence $x := \bigvee_n x_n$ is $\Phi$-compact with $x_0 \le x << y$; since $y = \bigvee \waydown y$, it follows that $X$ is $\Phi$-algebraic, proving \cref{thm:r-way:nomega:alg}.
\end{proof}

\begin{corollary}
\label{thm:cts-alg}
For a join doctrine $\Phi$, the following are equivalent:
\begin{enumerate}[label=(\roman*)]
\item  $\omega \not\in \Phi$.
\item  $\#I$ is $\Phi$-algebraic.
\item  Every $\Phi$-continuous lattice is $\Phi$-algebraic.
\qed
\end{enumerate}
\end{corollary}

\section{Commuting meets and joins}
\label{sec:comm}

We are interested in recovering $\Phi$-continuous lattices from their dual algebras of morphisms (to $2$ or $\#I$).
In order to do so, by general duality theory, the dual algebras must be equipped with all operations which commute with the $\Phi$-continuous lattice operations of arbitrary meets and $\Phi$-joins.
Thus, we now review the theory of classes of commuting meets and joins, again due in the general enriched categories context to \cite{KScolim}, although the posets case is much simpler.

It is convenient to treat a ``class of meets'' as simply the order-dual of a ``class of joins''.
Thus, given a join doctrine $\Phi$, we will refer to $\Phi^\op := \{\phi^\op \mid \phi \in \Phi\}$ as a \defn{meet doctrine}, and a meet indexed by $\phi^\op \in \Phi^\op$ as a \defn{$\Phi^\op$-meet}.
A poset with all $\Phi^\op$-meets is a \defn{$\Phi^\op$-inflattice}, with the category of all such denoted $\Phi^\op\!{Inf}$.
A \defn{$\Phi^\op$-filter} is an upper sub-$\Phi^\op$-inflattice.
The free $\Phi^\op$-inflattice generated by a poset $X$ is $\Phi(X^\op)^\op$.

\begin{definition}[see \cite{KScolim}]
\label{def:comm}
For two join doctrines $\Phi, \Psi$, where we regard $\Psi^\op$ as a meet doctrine, to say that \defn{$\Psi^\op$-meets commute with $\Phi$-joins in $2$} means that for any posets $X, Y$,
\begin{gather*}
\forall \phi \in \Phi(Y)\, \forall \psi \in \Psi(X)\, \forall F : X^\op \times Y -> 2\, \paren[\Big]{\bigwedge_{x \in \psi} \bigvee_{y \in \phi} F(x, y) = \bigvee_{y \in \phi} \bigwedge_{x \in \psi} F(x, y)}
\end{gather*}
(where $F$ runs over monotone maps).
By currying $F$, this is equivalent to
\begin{gather*}
\forall \phi \in \Phi(Y)\, \forall \psi \in \Psi(X)\, \forall f : Y -> \@L(X)\, \paren[\Big]{\psi \subseteq \bigcup_{y \in \phi} f(y) \iff \exists y \in \phi\, (\psi \subseteq f(y))} \\
\iff \forall \psi \in \Psi(X)\, (\psi \in \@L(X) \text{ is $\Phi$-compact}).
\end{gather*}
We write $\Phi^*(X) := \@L(X)_\Phi$ for the $\Phi$-compact lower sets $\psi \subseteq X$, i.e., those indexing meets commuting with $\Phi$-joins in $2$.
Note that by order-duality, the roles of $\Phi, \Psi$ may be swapped.
Thus
\begin{equation*}
\text{$\Psi^\op$-meets commute with $\Phi$-joins in $2$}
\iff \Psi \subseteq \Phi^*
\iff \Phi \subseteq \Psi^*
\quad \text{(as submonads of $\@L$)}.
\end{equation*}
\end{definition}

\begin{remark}
\label{rmk:comm-proper}
The above definition of $\Phi^*$, which follows \cite{KScolim}, yields \emph{a priori} a submonad of $\@L$.
But such a submonad automatically obeys the saturation condition \cref{rmk:phi-proper:proper} of \cref{rmk:phi-proper}, since given an order-embedding $i : X `-> X'$ and poset $Y$, a monotone $F : X^\op \times Y -> 2$ may be extended along $i$ to $F' : X'^\op \times Y -> 2$ (e.g., the left Kan extension $F'(x',y) := \bigvee_{x \in i^{-1}(\up x')} F(x,y)$), so that for $\psi \in \@L(X)$, the $\psi^\op$-meet of $F$ commutes with all $\Phi$-joins iff the $i_*(\psi)^\op$-meet of $F'$ does.
Thus by \cref{rmk:phi-proper}, we may equally well regard $\Phi^*$ as a class of posets.
Namely, for a poset $\psi$,
\begin{align*}
\psi \in \Phi^*
&\iff  \psi \in \Phi^*(\psi) = \@L(\psi)_\Phi \\
&\iff  \text{whenever $\psi$ is a $\Phi$-union of lower subsets, one of them is $\psi$}.
\end{align*}
Note moreover that this reasoning applies to $\Phi^*$ even if $\Phi$ is only a submonad of $\@L$ to begin with; this justifies our claim from \cref{rmk:phi-proper} that for our duality-theoretic purposes, it suffices to consider ``join doctrines'' which are classes of posets, rather than arbitrary submonads of $\@L$ as in \cite{KScolim}.
\end{remark}

\begin{remark}
\label{rmk:comm-r}
$\Phi$-joins commute with $\Psi^\op$-meets in $2$ iff they do in the unit interval $\#I$.
This follows from the facts that $2$ is a complete sublattice of $\#I$, while $\#I$ is a complete lattice homomorphic image via ${\bigvee} : \@L(\#I) ->> \#I$ (by complete distributivity, \cref{ex:cts-r}) of a complete sublattice $\@L(\#I) \subseteq 2^\#I$.
\end{remark}

\begin{remark}
\label{rmk:comm-idl}
If $\phi \in \Psi^*(X)$ for a $\Psi$-suplattice $X$, then by considering the indicator function of ${\le} \subseteq X^\op \times X$, we get that $\phi$ must be a $\Psi$-ideal.
(The converse is false in general: for $\Psi =$ directed posets, a $\Psi$-ideal is a Scott-closed subset, but only finite meets commute with directed joins.)
\end{remark}

\begin{proposition}[{\cite[8.9, 8.11, 8.13]{KScolim}}]
\label{thm:sound}
Let $\Phi, \Psi$ be two join doctrines such that $\Psi^\op$-meets commute with $\Phi$-joins in $2$.
The following are equivalent:
\begin{enumerate}[label=(\roman*)]
\item \label{thm:sound:alg}
For every poset $X$, $\@L(X)$ is generated under $\Phi$-joins by $\Psi(X) \subseteq \@L(X)_\Phi$.
\item \label{thm:sound:idl}
For every $\Psi$-suplattice $X$, $\Phi(X)$ consists precisely of all $\Psi$-ideals in $X$.
\item \label{thm:sound:idl-l}
For every poset $X$, there is a sub-$\Psi$-suplattice $\Psi'(X) \subseteq \@L(X)$ containing all principal ideals $\down x$ (e.g., $\Psi'(X) = \@L(X)$ or $\Psi'(X) = \Psi(X)$) such that $\Phi(\Psi'(X))$ contains all $\Psi$-ideals in $\Psi'(X)$.
\end{enumerate}
If these hold, then in fact $\Psi(X) = \@L(X)_\Phi = \Phi^*(X)$, whence $\@L(X) \cong \Phi(\Psi(X))$ is $\Phi$-algebraic, whence in particular $\Phi$ is a continuous join doctrine; and similarly $\Phi = \Psi^*$.
\end{proposition}

If these hold, we call $\Phi$ a \defn{sound join doctrine, dual to the sound meet doctrine $\Psi^\op$}.
Thus, $\Phi$ is a sound join doctrine iff $\@L(X) \cong \Phi(\Phi^*(X))$, iff $\Phi(X)$ contains every $\Phi^*$-ideal in a $\Phi^*$-suplattice $X$.
(Warning: this notion is \emph{not} preserved under swapping $\Phi, \Psi$, in contrast to \cref{def:comm}.)

\begin{proof}
\cref{thm:sound:idl}$\implies$\cref{thm:sound:idl-l} is obvious.

\cref{thm:sound:idl-l}$\implies$\cref{thm:sound:alg}:
For any $\theta \in \@L(X)$, clearly $\Psi'(X) \cap \down \theta = \{\psi \in \Psi'(X) \mid \psi \subseteq \theta\}$ is a $\Psi$-ideal in $\Psi'(X)$, thus by \cref{thm:sound:idl-l} is in $\Phi(\Psi'(X))$; and its union is $\theta$, which is thus a $\Phi$-join of elements of $\Psi(X)$.

\cref{thm:sound:alg}$\implies$\cref{thm:sound:idl}:
For every $\theta \in \@L(X)$, the $\Psi$-ideal $\ang{\theta}$ it generates is in $\Phi(X)$: this is true for $\theta \in \Psi(X)$ since $\ang{\theta} = \down \bigvee \theta$, and is true for a $\Phi$-join $\theta = \bigcup_i \theta_i$ if it is true for each $\theta_i$ since $\ang{\theta} = \bigcup_i \ang{\theta_i}$ (using that $\Psi^\op$-meets commute with $\Phi$-joins in $2$), thus is true for all $\theta \in \@L(X)$ by \cref{thm:sound:alg}.
Conversely, as noted above, every $\phi \in \Phi(X)$ is a $\Psi$-ideal.

The last sentence follows from \cref{thm:sound:alg}, \cref{thm:sound:idl}, and \cref{rmk:comm-idl}, which imply that $\Phi(X) = \Psi^*(X)$ for a $\Psi$-suplattice $X$, hence for every poset $X$ by applying \cref{rmk:phi-proper:proper} in \cref{rmk:phi-proper} to $\down : X -> \Psi(X)$.
\end{proof}

\begin{lemma}
\label{rmk:sound-omega}
For any join doctrine $\Phi$, we have $\omega \in \Phi$ iff $\omega \not\in \Phi^*$.
\end{lemma}
\begin{proof}
$\omega \not\in \Phi \cap \Phi^*$ since $\omega$-joins do not commute with $\omega^\op$-meets in $2$.
If $\omega \not\in \Phi^*$, i.e., $\omega \in \@L(\omega)$ is not $\Phi$-compact, then $\omega$ is a $\Phi$-union of proper lower subsets of $\omega$; the order-type of this union must clearly be $\omega$.
(This argument is due to the referee; my original proof assumed soundness of $\Phi$.)
\end{proof}

\begin{corollary}[generalized Hofmann--Mislove--Stralka duality]
\label{thm:hms}
Let $\Phi$ be a sound join doctrine, dual to the meet doctrine $\Psi^\op = \Phi^{*\op}$.
We have a dual equivalence of categories
\begin{equation*}
\begin{tikzcd}[row sep=0pt, column sep=6em]
\Phi\!{AlgLat}^\op \rar[shift left, "{\Phi\!{AlgLat}(-, 2)}"] &
\Psi^\op\!{Inf}. \lar[shift left, "{\Psi^\op\!{Inf}(-, 2)}"]
\end{tikzcd}
\end{equation*}
We may replace $\Phi\!{AlgLat}$ with $\Phi\!{CtsLat}$ iff $\omega \not\in \Phi$, i.e., $\omega \in \Psi$.
\end{corollary}
\begin{proof}
For a $\Phi$-algebraic lattice $X$, a morphism $X -> 2$ is the indicator function of $\up x$ for $\Phi$-algebraic $x$.
For a $\Psi^\op$-inflattice $A$, a morphism $A -> 2$ is the indicator function of a $\Psi^\op$-filter.
So we have
\begin{align*}
\Phi\!{AlgLat}(X, 2) &\cong X_\Phi^\op, &
\Psi^\op\!{Inf}(A, 2) &\cong \Phi(A^\op).
\end{align*}
Now the adjunction (co)unit on the left is given by, for $X \in \Phi\!{AlgLat}$, the evaluation map
\begin{align*}
X &--> \Psi^\op\!{Inf}(\Phi\!{AlgLat}(X, 2), 2) \\
x &|--> (f |-> f(x)),
\end{align*}
which via the above isomorphisms becomes the canonical isomorphism
$X \cong \Phi(X_\Phi)$
characterizing algebraicity.
Similarly, for $A \in \Psi^\op\!{Inf}$, the unit $A -> \Phi\!{AlgLat}(\Psi^\op\!{Inf}(A, 2), 2)$ is the canonical isomorphism
$
A^\op \cong \Phi(A^\op)_\Phi.
$
By \cref{thm:cts-alg}, $\Phi\!{AlgLat} = \Phi\!{CtsLat}$ iff $\#I$ is $\Phi$-algebraic, iff $\omega \not\in \Phi$.
\end{proof}

\begin{example}
$\Phi =$ directed posets forms a sound join doctrine, dual to $\Psi^\op =$ ``finite meets'', i.e., $\Psi =$ the class of posets with finite cofinality.
In this case, \cref{thm:hms} becomes the classical Hofmann--Mislove--Stralka duality \cite{HMSlat} between (unital) meet-semilattices and algebraic lattices.

Similarly, the join doctrine $\Phi$ of $\kappa$-directed posets for an uncountable regular cardinal $\kappa$ is sound, dual to $\kappa$-ary meets.
But since $\omega \not\in \Phi$ for uncountable $\kappa$, we get a duality between $\kappa$-meet-semilattices and $\kappa$-\emph{continuous} lattices.
\end{example}

We now show that there are very few sound join doctrines $\Phi \ni \omega$, for which $\Phi\!{AlgLat} \ne \Phi\!{CtsLat}$: essentially, they are only the classical cases of continuous and completely distributive lattices (\cref{ex:cts-dir,ex:cts-cd}), plus the minor variations including/excluding empty joins.

\begin{theorem}
\label{thm:sound-omega}
There are precisely 4 sound join doctrines $\Phi \ni \omega$, dual to $\Psi^\op$:
\begin{enumerate}[label=(\roman*)]
\item  $\Phi =$ directed posets, $\Psi =$ posets with finite cofinality;
\item  $\Phi =$ empty or directed posets, $\Psi =$ nonempty posets with finite cofinality;
\item  $\Phi =$ nonempty posets, $\Psi =$ posets which are empty or have greatest element;
\item  $\Phi =$ all posets, $\Psi =$ posets with greatest element.
\end{enumerate}
\end{theorem}
\begin{proof}
It is well-known and easily seen that each of these 4 cases is sound; we show the converse.

First, we show that $\Phi$ must contain every directed poset, i.e., every poset in $\Psi$ must have finite cofinality.
For every set $X$, $\Phi$ contains the finite powerset $\@P_\omega(X)$, since this is a $\Psi$-ideal in the full powerset $\@P(X)$, since by \cref{thm:r-way}\cref{thm:r-way:nomega:seq} (applied to $\Psi \not\ni \omega$), every $\psi \in \Psi(\@P_\omega(X))$ can have neither a strictly increasing sequence nor infinitely many maximal elements, thus must be finite.
Now for every join-semilattice $X$, we have a monotone surjection $\bigvee : \@P_\omega(X) ->> X$, whence $X \in \Phi$.
Since every directed poset $\phi$ is cofinal in the free join-semilattice it generates, it follows that $\phi \in \Phi$.

So $\Psi$ is determined by the finite antichains $n$ in it.
If some $n > 1$ is in $\Psi$, then by induction so is each $n^k \cong \bigsqcup_{i \in n} n^{k-1}$; now every $m \ge 1$ admits a surjection $n^k ->> m$, whence $m \in \Psi$.
\end{proof}

\section{$\#U$-posets}
\label{sec:upos}

Henceforth, we assume $\Phi \ni \omega$ is a sound join doctrine, dual to $\Psi^\op$, so one of the cases in \cref{thm:sound-omega}.
Then Hofmann--Mislove--Stralka duality does not apply to all $\Phi$-continuous lattices, and so we would like to formulate a duality based on morphisms to $\#I$ instead of $2$.

By \cref{rmk:comm-r}, the dual algebra $\Phi\!{CtsLat}(X, \#I)$ will still be equipped with $\Psi^\op$-meets.
But these are not all the operations on $\#I$ commuting with the $\Phi$-continuous lattice operations: clearly any complete lattice homomorphism $\#I -> \#I$ does as well.
We thus introduce the following notions:

\begin{definition}
\label{def:upos}
Let $\#U := \!{CLat}(\#I, \#I)$ denote the partially ordered monoid of all complete lattice homomorphisms $\#I -> \#I$, i.e., surjective monotone maps.

A \defn{$\#U$-poset} is a poset equipped with a monotone (in both variables) action of the monoid $\#U$.
Denote the category of these (and equivariant monotone maps) by $\#U\!{Pos}$.

A \defn{$\#U$-$\Psi^\op$-inflattice} is a $\#U$-poset which is also a $\Psi^\op$-inflattice such that the action of each $u \in \#U$ preserves $\Psi^\op$-meets.
Denote the category of these by $\#U\Psi^\op\!{Inf}$.
\end{definition}

\begin{definition}
\label{def:upos-met}
Let $\dotplus, \dotminus$ denote truncated $+, -$ on $\#I$;
note that they obey the adjunction
\begin{equation}
\label{eq:dotpm-adj}
r \dotminus s \le t \iff r \le s \dotplus t.
\end{equation}
For a $\#U$-poset $A$ and $a, b \in A$, define
\begin{gather*}
\begin{aligned}
a \le_r b  \coloniff{}&  \forall u, v \in \#U\, (u((-) \dotplus r) \le v \implies u(a) \le v(b)),
\end{aligned}
\\
\begin{aligned}
\rho(a, b) &:= \bigwedge \{r \in \#I \mid a \le_r b\}, \\
\dist(a, b) &:= \rho(a, b) \vee \rho(b, a).
\end{aligned}
\end{gather*}
\end{definition}

\begin{remark}
\label{rmk:upos-met}
In the definition of $\le_r$, instead of testing $\forall u, v$, it is enough to test any particular $u \in \#U$ which restricts to an order-isomorphism $u : [r,1] \cong [0,1]$ (e.g., the linear such isomorphism extended by $0$ on $[0,r]$), so that $v := u((-) \dotplus r) \in \#U$.
Indeed, for any other $u', v' \in \#U$ with $u'((-) \dotplus r) \le v'$, there is $w \in \#U$ with $u' = w \circ u$, whence $u'(a) = w(u(a)) \le w(v(a)) \le v'(a)$.
\end{remark}

\begin{remark}
There is an evident order-duality for $\#U$-posets $A$: let $u \in \#U$ act on the order-dual $A^\op$ via $1-u(1-(-))$; this reverses each $\le_r$, and turns $\rho$ into $\rho^\op(a, b) := \rho(b, a)$.
\end{remark}

Intuitively, $a \le_r b$ means ``$a \le b \dotplus r$''.
The following properties justify this interpretation:

\begin{proposition}
\label{thm:upos-met-r}
In $\#I$, we have $a \le_r b \iff a \le b \dotplus r$, whence $\rho(a, b) = a \dotminus b$ and $\dist(a, b) = \abs{a-b}$.
\end{proposition}
\begin{proof}
If $a \le b \dotplus r$, then for every $u, v \in \#U$ with $u((-) \dotplus r) \le v$, we have $u(a) \le u(b \dotplus r) \le v(b)$.

For the converse, the case $r = 1$ is vacuous; thus we may assume $r < 1$.
Note that $(-) \dotplus r : \#I -> \#I$ can be written as $u^\times \circ v$ where $v := 1 \wedge (-)/(1-r)$, $u := v((-) \dotminus r)$, and $u^\times$ is the right adjoint of $u$.
Now from $a \le_r b$ and $u((-) \dotplus r) = v$, we get $u(a) \le v(b)$, whence $a \le u^\times(v(b)) = b \dotplus r$.
\end{proof}

\begin{lemma}
\label{thm:upos-met}
In every $\#U$-poset $A$, we have the following, for $r, s, t \in \#I$, $u, v \in \#U$, $a, b, c \in A$:
\begin{enumerate}[label=(\alph*)]
\item \label{thm:upos-met:mono}
$r \le s \AND a \le_r b \implies a \le_s b$.
\item \label{thm:upos-met:refl}
$\le_0$ is the same as $\le$.
\item \label{thm:upos-met:trans}
$a \le_r b \le_s c \implies a \le_{r \dotplus s} c$.
\item \label{thm:upos-met:met}
$\rho$ is a pseudoquasimetric: $\rho(a, a) = 0$, and $\rho(a, b) + \rho(b, c) \ge \rho(a, c)$.
Thus, $\dist$ is a pseudometric.
\item \label{thm:upos-met:act}
$u((-) \dotplus r) \le v \dotplus s \AND a \le_r b \implies u(a) \le_s v(b)$.
Thus, $\rho(u(a), v(a)) \le \rho(u, v) := \bigvee (u \dotminus v)$, i.e., the $\#U$-action is 1-Lipschitz in the first variable with respect to the $\ell^\infty$-quasimetric on $\#U$.
Moreover, if $u \in \#U$ is uniformly continuous with modulus $\mu : \#I -> \#I$, i.e., $u(r) \dotminus u(s) \le \mu(r \dotminus s)$, then the action of $u$ is uniformly continuous with the same modulus: $\rho(u(a), u(b)) \le \mu(\rho(a, b))$.
\item \label{thm:upos-met:radj}
$u^\times((-) \dotplus r) \le v \dotplus s \AND u(a) \le_r b \implies a \le_s v(b)$ (where $u^\times$ is the right adjoint of $u$).
\end{enumerate}
In a $\#U$-$\Psi^\op$-inflattice, we moreover have, for $\psi, \psi' \in \Psi(A^\op)$:
\begin{enumerate}[resume*]
\item \label{thm:upos-met:inf}
$a \le_r \bigwedge \psi \iff \forall b \in \psi\, (a \le_r b)$.
Thus,
$\rho(\bigwedge \psi, \bigwedge \psi') \le \bigwedge_{a \in \psi} \bigvee_{b \in \psi'} \rho(a, b)$.
\end{enumerate}
\end{lemma}
\begin{proof}
\cref{thm:upos-met:mono} and \cref{thm:upos-met:refl} are straightforward, as is \cref{thm:upos-met:met} given the previous parts.

\cref{thm:upos-met:trans}
For $u, w \in \#U$ with $u((-) \dotplus (r \dotplus s)) \le w$, we have $v := u((-) \dotplus r) \in \#U$ with $u((-) \dotplus r) \le v$ and $v((-) \dotplus s) \le w$, whence $u(a) \le v(b) \le w(c)$.

\cref{thm:upos-met:act}
For $u', v' \in \#U$ with $u'((-) \dotplus s) \le v'$, we have
$u'(u((-) \dotplus r)) \le u'(v(-) \dotplus s) \le v' \circ v$,
whence $u'(u(a)) \le v'(v(b))$.
For the last assertion: $u(r) \dotminus u(s) \le \mu(r \dotminus s)$ means $u((-) \dotplus r) \le u(-) \dotplus \mu(r)$.

\cref{thm:upos-met:radj}
The assumption is equivalent to $(-) \dotminus s \le v(u(-) \dotminus r)$;
thus for $u', v' \in \#U$ with $u'((-) \dotplus s) \le v'$, we have
$u' \le v'((-) \dotminus s) \le v'(v(u(-) \dotminus r))$,
whence $u'(a) \le v'(v(u(a) \dotminus r)) \le v'(v(b))$.

\cref{thm:upos-met:inf}
$\Longrightarrow$ and the last assertion follow from \cref{thm:upos-met:trans}.
For $\Longleftarrow$: for $u, v \in \#U$ with $u((-) \dotplus r) \le v$, we have
$u(a) \le \bigwedge_{b \in \psi} v(b) = v(\bigwedge \psi)$.
\end{proof}

For general background on (pseudo)quasimetrics, see e.g., \cite{Kqunif}.
A pseudoquasimetric $\rho$ as above induces a topology, where a basic neighborhood of $a \in A$ is $\{b \in A \mid \rho(a, b) < r\}$ for some $r > 0$.
Thus the closure of $B \subseteq A$ is the set of all $a \in A$ such that
\begin{align*}
\rho(a, B) = \bigwedge_{b \in B} \rho(a, b) = 0,
\end{align*}
which is in particular a lower set.
To avoid confusion, we will call a closed set in this topology a \defn{$\rho$-closed lower set}, and denote the set of all such by $\-{\@L}(A) \subseteq \@L(A)$.
We will also say \defn{$\rho^\op$-closed upper set} $B \subseteq A$ for the order-dual notion, i.e., if $\rho(B, a) = 0$ then $a \in B$; the set of all such is thus $\-{\@L}(A^\op)$.
For a $\#U$-$\Psi^\op$-inflattice $A$, recalling that $\Phi(A^\op)$ consists of $\Psi^\op$-filters by soundness, let
\begin{equation*}
\-\Phi(A^\op) := \Phi(A^\op) \cap \-{\@L}(A^\op)
\end{equation*}
denote the \defn{$\rho^\op$-closed $\Psi^\op$-filters} in $A$.

\begin{lemma}
If $\phi \in \Phi(A^\op)$ is a $\Psi^\op$-filter, then so is the $\rho^\op$-closure $\-\phi$.
\end{lemma}
\begin{proof}
This follows from the facts that $\Psi^\op$ is a class of finite meets by \cref{thm:sound-omega}, and that $\Psi^\op$-meets are Lipschitz by \cref{thm:upos-met}\cref{thm:upos-met:inf}.
\end{proof}

As usual for actions, a subset $B \subseteq A$ of a $\#U$-poset is \defn{$\#U$-invariant} if it is closed under the action.
For a class of sets $\Gamma(A)$, we write $\Gamma^\#U(A)$ for the $\#U$-invariant members, e.g., $\@L^\#U(A), \-\Phi^\#U(A)$.

\begin{lemma}
\label{thm:upos-uclfilt}
If $\phi \in \@P^\#U(A)$ is a $\#U$-invariant filter base, then its $\rho^\op$-closure $\-\phi$ is a $\#U$-invariant $\Psi^\op$-filter, hence is the $\#U$-invariant $\rho$-closed $\Psi^\op$-filter generated by $\phi$.
\end{lemma}
\begin{proof}
By uniform continuity of the action of each $u$ (\cref{thm:upos-met}\cref{thm:upos-met:act}), $\-\phi$ is $\#U$-invariant.
It is also upper, since every $\rho^\op$-closed set is, thus it is also the $\rho^\op$-closure of the upward closure of $\phi$, which is a $\Psi^\op$-filter since $\Psi^\op$-meets are finite by \cref{thm:sound-omega}, whence so is $\-\rho$ by the preceding lemma.
\end{proof}

\begin{proposition}
\label{thm:upos-ker}
For a $\#U$-$\Psi^\op$-inflattice $A$, we have an order-isomorphism
\begin{align*}
\#U\Psi^\op\!{Inf}(A, \#I) &\cong \-\Phi^\#U(A^\op) = \{\text{$\#U$-invariant $\rho^\op$-closed $\Psi^\op$-filters in } A\} \\
f &|-> f^{-1}(1) \\
1 - \rho(\phi, -) &<-| \phi.
\end{align*}
\end{proposition}
\begin{proof}
For ease of notation, we will prove the dual statement that for a $\#U$-$\Psi$-suplattice $A$,
\begin{align*}
\#U\Psi\!{Sup}(A, \#I)^\op &\cong \-\Phi^\#U(A) = \{\text{$\#U$-invariant $\rho$-closed $\Psi$-ideals in } A\} \\
f &|-> f^{-1}(0) \\
\rho(-, \phi) &<-| \phi.
\end{align*}
It is immediate from the definitions that for a $\#U$-equivariant $\Psi$-join-preserving $f : A -> \#I$, $f^{-1}(0) \subseteq A$ is $\#U$-invariant $\rho$-closed lower, and also that a $\rho$-closed lower $\phi \subseteq A$ is equal to $\rho(-, \phi)^{-1}(0)$.

We now check that for a $\#U$-invariant $\Psi$-ideal $\phi \subseteq A$, $\rho(-, \phi) : A -> \#I$ is $\#U$-equivariant $\Psi$-join-preserving (it is clearly monotone).
For $\psi \in \Psi(A)$,
\begin{equation*}
\begin{aligned}
\rho(\bigvee \psi, \phi)
&= \bigwedge_{b \in \phi} \bigvee_{a \in \psi} \rho(a, b) &&\text{by the dual of \cref{thm:upos-met}\cref{thm:upos-met:inf}} \\
&= \bigvee_{a \in \psi} \bigwedge_{b \in \phi} \rho(a, b) &&\text{because $\Phi \subseteq \Psi^*$ (\cref{rmk:comm-r})} \\
&= \bigvee_{a \in \psi} \rho(a, \phi);
\end{aligned}
\end{equation*}
thus $\rho(-, \phi)$ preserves $\Psi$-joins.
To check $\#U$-equivariance: let $u \in \#U$ and $a \in A$.
We have
\begin{align*}
\rho(u(a), \phi) &= \bigwedge_{b \in \phi} \rho(u(a), b) = \bigwedge \{r \in \#I \mid u(a) \le_r b \in \phi\}, \\
u(\rho(a, \phi)) &= u(\bigwedge_{b \in \phi} \rho(a, b)) = \bigwedge_{b \in \phi} u(\rho(a, b)) = \bigwedge \{u(r) \mid a \le_r b \in \phi\}.
\end{align*}
For each $a \le_r b \in \phi$, find
\begin{equation*}
u((-) \dotplus r) \dotminus u(r) \le v \in \#U,
\end{equation*}
whence
$u(a) \le_{u(r)} v(b) \in \phi$
by \cref{thm:upos-met}\cref{thm:upos-met:act}; this proves $u(\rho(a, \phi)) \ge \rho(u(a), \phi)$.
Conversely, for $u(a) \le_r b \in \phi$ with $r < 1$, let $u^\times$ be the right adjoint of $u$, and similarly to before, find
\begin{equation*}
u^\times((-) \dotplus r) \dotminus u^\times(r) \le v \in \#U,
\end{equation*}
whence
$a \le_{u^\times(r)} v(b) \in \phi$
by \cref{thm:upos-met}\cref{thm:upos-met:radj}, whence
$u(\rho(a, \phi)) \le r$;
so $\rho(u(a), \phi) \ge u(\rho(a, \phi))$.

Finally, we check that for $\#U$-equivariant monotone $f : A -> \#I$, we have $f = \rho(-, f^{-1}(0))$.
We have $\le$ since $f$ is $1$-Lipschitz.
Conversely, for $a \in A$ with $f(a) < 1$, find
$(-) \dotminus f(a) \le u \in \#U$
with
$u(f(a)) = 0$;
then $a \le_{f(a)} u(a)$ by \cref{thm:upos-met}\cref{thm:upos-met:act}, so $\rho(a, f^{-1}(0)) \le \rho(a, u(a)) \le f(a)$.
\end{proof}

The $\#U$-poset $\#I$ obeys the following additional axioms, which must thus also hold in the dual of a $\Phi$-continuous lattice:

\begin{definition}
\label{def:upos-arch-compl}
We call a $\#U$-poset $A$ \defn{Archimedean} if it obeys
\begin{equation*}
\forall r > 0\, (a \le_r b)  \implies  a \le b.
\end{equation*}
We call $A$ \defn{(Cauchy-)complete} if it is Archimedean and also complete in the metric $\dist$.
\end{definition}

\begin{definition}
\label{def:upos-stack}
We call a $\#U$-poset $A$ \defn{unstackable} if for any $0 < r < 1$ and $u, v \in \#U$ restricting to order-isomorphisms $u : [0,r] \cong [0,1]$ and $v : [r,1] \cong [0,1]$, we have
\begin{equation*}
u(a) \le u(b) \AND v(a) \le v(b)  \implies  a \le b.
\end{equation*}
We call $A$ \defn{stackable} if it is unstackable and for $r, u, v$ as above and $a, b \in A$ such that $v'(b) \le u'(a)$ for all $u', v' \in \#U$, there is a (unique, by unstackability) $c \in A$ with $u(c) = a$ and $v(c) = b$.
\end{definition}

Intuitively, stackability means that, thinking of $A$ as the dual of a $\Phi$-continuous lattice $X$, we may specify $A \ni a : X -> \#I$ via its restrictions to its sublevel and superlevel sets $a^{-1}([0,r]), a^{-1}([r,1])$.

\begin{remark}
\label{rmk:upos-stack}
As in \cref{rmk:upos-met}, it is enough to take some particular $u, v$ above.
Also, it is enough to take some particular $r$ (e.g., $1/2$), since we may move $r$ around via an order-isomorphism $\#I \cong \#I$.
\end{remark}

\begin{lemma}
\label{thm:upos-stack-n}
If $A$ is (un)stackable, then more generally, for $0 = r_0 < r_1 < \dotsb < r_n = 1$ and $u_1, \dotsc, u_n \in \#U$ restricting to $u_i : [r_{i-1},r_i] \cong [0,1]$, for $a_1, \dotsc, a_n \in A$ such that $v'(a_{i+1}) \le u'(a_i)$ for all $u', v' \in \#U$, there is (at most one, depending monotonically on $(a_1, \dotsc, a_n)$) $a \in A$ with $u_i(a) = a_i$.
\end{lemma}
\begin{proof}
By a straightforward induction on $n$.
\end{proof}

\begin{lemma}
\label{thm:upos-unstack-r}
If $A$ is unstackable, then more generally, for $0 \le r = r_0 < r_1 < \dotsb < r_n = 1$ and $u_1, \dotsc, u_n \in \#U$ with $u_i : [r_{i-1},r_i] \cong [0,1]$, so that $u_i((-) \dotplus r) \in \#U$, for any $a, b \in A$, we have
\begin{equation*}
u_1(a) \le u_1(b \dotplus r) \AND \dotsb \AND u_n(a) \le u_n(b \dotplus r)  \implies  a \le_r b.
\end{equation*}
\end{lemma}
\begin{proof}
By \cref{rmk:upos-met}, it suffices to check that for $w \in \#U$ with $w : [r,1] \cong [0,1]$, we have $w(a) \le w(b \dotplus r)$;
this follows from applying the preceding lemma to $u_i \circ w^{-1} : [w(r_{i-1}),w(r_i)] \cong [0,1]$.
\end{proof}

\section{The duality}
\label{sec:dual}

Let $\!{CSt}\#U\Psi^\op\!{Inf} \subseteq \#U\Psi^\op\!{Inf}$ denote the full subcategory of complete stackable $\#U$-$\Psi^\op$-inflattices.
Since the $\Phi$-continuous lattice and $\#U$-$\Psi^\op$-inflattice structures on $\#I$ commute, we have a dual adjunction
\begin{equation}
\label{eq:dual-adj}
\begin{tikzcd}[row sep=0pt, column sep=6em]
\Phi\!{CtsLat}^\op \rar[shift left, "{\Phi\!{CtsLat}(-, \#I)}"] &
\!{CSt}\#U\Psi^\op\!{Inf} \subseteq \#U\Psi^\op\!{Inf}. \lar[shift left, "{\#U\Psi^\op\!{Inf}(-, \#I)}"]
\end{tikzcd}
\end{equation}

\begin{theorem}
\label{thm:cts-dual}
For every $\Phi$-continuous lattice $X$, the evaluation map
\begin{align*}
\eta : X &--> \#U\Psi^\op\!{Inf}(\Phi\!{CtsLat}(X, \#I), \#I) \\
x &|--> (f |-> f(x)),
\end{align*}
which is the (co)unit on the left side of the above adjunction, is an order-isomorphism.
\end{theorem}
\begin{proof}
Via \cref{thm:cts-ladj,thm:upos-ker}, $\eta$ corresponds to the map
\begin{alignat*}{3}
\~\eta : X &--> \-\Phi^\#U({<<^\Phi}\!{Sup}(\#I, X)) \subseteq \@L({<<^\Phi}\!{Sup}(\#I, X)) \\
x &|--> \{f^+ \in {<<^\Phi}\!{Sup}(\#I, X) \mid f^+(1) \le x\}
\end{alignat*}
whose left adjoint is easily seen to be
\begin{align*}
\~\eta^+ : \@L({<<^\Phi}\!{Sup}(\#I, X)) &--> X \\
\phi &|--> \bigvee_{f^+ \in \phi} f^+(1).
\end{align*}
That $x \le \~\eta^+(\~\eta(x))$ is Urysohn's lemma for $\Phi$-continuous lattices; see \cite[IV-3.1, IV-3.32]{GHKLMS}, \cite[VII~1.14, 3.2]{Jstone}, \cite{Xctslat}.
Since $x = \bigvee \waydown x$, it suffices to show that for each $y << x$ there is $f^+ \in {<<^\Phi}\!{Sup}(\#I, X)$ with $y \le f^+(1) \le x$.
Let $\#I_2 \subseteq \#I$ be the dyadic rationals, define $g : \#I_2 -> X$ by $g(0) := y$, $g(1) := x$, and inductively using interpolation (\cref{thm:way-props}\cref{thm:way-props:interp}) so that $r < s \implies g(r) << g(s)$; then $f^+(r) := \bigvee g(\#I_2 \cap [0,r))$ works.

Now let $\phi \in \-\Phi^\#U({<<^\Phi}\!{Sup}(\#I, X))$; we must show $\~\eta(\~\eta^+(\phi)) \subseteq \phi$.
Since $\~\eta$ preserves $\Phi$-joins,
\begin{align*}
\~\eta(\~\eta^+(\phi)) = \bigvee_{f^+ \in \phi} \~\eta(f^+(1)).
\end{align*}
For each $f^+ \in \phi$ and $g^+ \in \~\eta(f^+(1))$, i.e., $g^+(1) \le f^+(1)$, we have $1 \le g(f^+(1))$, thus there is $g \circ f^+ \ge u \in \#U$, whence $g \ge u \circ f$, so $g^+ \le (u \circ f)^+ \in \phi$ since $\phi$ is $\#U$-invariant; thus $\~\eta(f^+(1)) \subseteq \phi$.
\end{proof}

\begin{theorem}
\label{thm:uinf-dual}
For every Archimedean unstackable $\#U$-$\Psi^\op$-inflattice $A$, the evaluation map
\begin{align*}
\iota : A &--> \Phi\!{CtsLat}(\#U\Psi^\op\!{Inf}(A, \#I), \#I) \\
a &|--> (f |-> f(a))
\end{align*}
is an embedding.
If $A$ is stackable, its image is dense; thus if $A$ is also complete, it is an isomorphism.
\end{theorem}
\begin{proof}
Via \cref{thm:cts-ladj,thm:upos-ker}, $\iota$ corresponds to the map
\begin{align*}
\~\iota : A &--> {<<^\Phi}\!{Sup}(\#I, \-\Phi^\#U(A^\op))^\op \\
a &|--> (r |-> \min \{\phi \in \-\Phi^\#U(A^\op) \mid r \le 1 - \rho(\phi, a)\}).
\end{align*}
We claim that in fact, for $r > 0$, $\~\iota(a)(r)$ is the $\rho^\op$-closure $\smash{\-{U_r(a)}}$ of
\begin{equation*}
U_r(a) := \{u(a) \mid u \in \#U \AND u(r) = 1\}.
\end{equation*}
$\smash{\-{U_r(a)}}$ is a $\#U$-invariant $\Psi^\op$-filter by \cref{thm:upos-uclfilt}.
Each $u(a) \in \smash{U_r}(a)$ is in each $\phi \in \smash{\-\Phi^\#U(A^\op)}$ with $r \le 1 - \rho(\phi, a)$: if $u(s) = 1$ for some $s < r$, we may let $b \in \phi$ with $b \le_{1-s} a$ to get $\phi \ni u(b \dotminus (1-s)) \le u(a)$, while if there is no such $s$, we may write $u$ as a limit of $u_0, u_1, \dotsc$ for which there are such $s$, then use that $\phi$ is closed.
And $r \le 1 - \rho(\smash{\-{U_r(a)}}, a)$: letting
$(-) \dotplus (1-r) \ge u \in \#U$ with $u(r) = 1$,
we have $U_r(a) \ni u(a) \le_{1-r} a$ by \cref{thm:upos-met}\cref{thm:upos-met:act}.
This proves the claim.

Now to show that $\~\iota$ is an order-embedding: let $\~\iota(a) \ge \~\iota(b) : \#I -> \-\Phi^\#U(A^\op)$, i.e., $\smash{\-{U_r(a)}} \supseteq U_r(b)$ for every $r > 0$;
since $A$ is Archimedean, it suffices to show $a \le_{2/n} b$ for all $n \ge 3$.
For $i = 1, \dotsc, n$, let
\begin{equation}
\tag{$*$}
\label{thm:uinf-dual:v_i}
v_i \in \#U, \quad
v_i : [(i-1)/n,i/n] \cong [0,1].
\end{equation}
Then $v_i(b) \in U_{i/n}(b)$, so there is $u_i \in \#U$ with $u_i(i/n) = 1$ such that 
\begin{equation*}
u_i(a) \le_{1/n} v_i(b).
\end{equation*}
Let $u', v' \in \#U$ with $u'((-) \dotplus 1/n) \le v'$; then for $2 \le i \le n-1$, we have
$v_{i+1}(a) \le u'(u_i(a)) \le v'(v_i(b)) \le v_{i+1}(b \dotplus 2/n)$
since $v_{i+1}(i/n) = 0$, $u'(u_i(i/n)) = 1$, $v'(v_i((i-1)/n)) = 0$, and $v_{i+1}((i-1)/n \dotplus 2/n) = 1$.
Thus since $A$ is unstackable, by \cref{thm:upos-unstack-r} we have $a \le_{2/n} b$.

Finally, suppose $A$ is stackable, and let $f^+ \in {<<^\Phi}\!{Sup}(\#I, \-\Phi^\#U(A^\op))$, left adjoint to $f$;
we will find, for every $n \ge 2$, some $a \in A$ with $\dist(\iota(a), f) \le 2/n$.
For $i = 1, \dotsc, n$, we have
$f^+((i-1)/n) << f^+(i/n) = \bigvee_{a \in f^+(i/n)} \-{U_1(a)} = \bigvee_{a \in f^+(i/n)} \bigvee_{r < 1} \-{U_r(a)}$
(again by \cref{thm:upos-uclfilt}), whence
\begin{equation*}
f^+((i-1)/n) \subseteq \-{U_{r_i}(a_i)}
\end{equation*}
for some $a_i \in f^+(i/n)$ and $r_i < 1$.
Let $u_i \in \#U$ with $u_i(r_i) = 0$, and let $v_i$ as in \cref{thm:uinf-dual:v_i}.
Then for $u' \in \#U$,
\begin{equation*}
f^+((i-1)/n) \subseteq \up u'(u_i(a_i)) \subseteq \-{U_1(u_i(a_i))},
\end{equation*}
since for $b \in f^+((i-1)/n) \subseteq \smash{\-{U_{r_i}(a_i)}}$, for every $s > 0$, there is $u'' \in \#U$ with $u''(r_i) = 1$, whence $u' \circ u_i \le u''$, such that $u'(u_i(a_i)) \le u''(a_i) \le_s b$, whence $u'(u_i(a_i)) \le b$ since $A$ is Archimedean.
In particular, this holds for $b = v'(u_{i-1}(a_{i-1}))$ for every $v' \in \#U$, so by \cref{thm:upos-stack-n}, there is $a \in A$ with
\begin{equation*}
v_i(a) = u_i(a_i)
\end{equation*}
for each $i$.
Then
\begin{equation*}
U_{i/n}(a) = U_1(v_i(a)) = U_1(u_i(a_i)),
\end{equation*}
since every $u \in \#U$ with $u(i/n) = 1$ is $\ge u' \circ v_i$ for some $u' \in \#U$.
We now show that $\dist(f, \iota(a)) \le 2/n$,
in terms of the left adjoints $f^+, \~\iota(a)$:
for each $t \in \#I$, letting $1 \le i \le n$ with $t \le i/n \le t \dotplus 1/n$,
\begin{align*}
\~\iota(a)(t) = \-{U_t(a)}
&\subseteq \-{U_{i/n}(a)}
= \-{U_1(u_i(a_i))}
\subseteq f^+(i/n)
\subseteq f^+(t \dotplus 1/n), \\
f^+(t \dotminus 1/n)
&\subseteq f^+((i-1)/n)
\subseteq \-{U_1(u_i(a_i))}
= \-{U_{i/n}(a)}
\subseteq \-{U_{t \dotplus 1/n}(a)} = \~\iota(a)(t \dotplus 1/n).
\qedhere
\end{align*}
\end{proof}

\begin{theorem}
The dual adjunction \cref{eq:dual-adj} is a dual equivalence of categories between $\Phi$-continuous lattices and complete stackable $\#U$-$\Psi^\op$-inflattices.
\qed
\end{theorem}

It is worth explicitly restating the duality for the two main examples of $\Phi$:

\begin{corollary}
\label{thm:cdlat-dual}
Hom into $\#I$ yields a dual equivalence of categories between completely distributive lattices and complete stackable $\#U$-posets.
\qed
\end{corollary}

Let us say that a \defn{$\#U$-meet-semilattice} is a $\#U$-poset with finite meets preserved by the $\#U$-action.

\begin{corollary}
\label{thm:ctslat-dual}
Hom into $\#I$ yields a dual equivalence of categories between continuous lattices and complete stackable $\#U$-meet-semilattices.
\qed
\end{corollary}

We end by showing that in the presence of meets, stackability admits a simpler formulation:

\begin{definition}
Let $\smash{\^{\#U}} := \!{CtsLat}(\#I, \#I) \supseteq \#U$ be the monoid of continuous lattice morphisms $\#I -> \#I$, i.e., continuous monotone maps preserving $1$, but possibly not $0$.

A \defn{$\^{\#U}$-module} is a (unital) meet-semilattice with a $\^{\#U}$-action preserving finite meets on both sides.
\end{definition}

\begin{proposition}
\label{thm:umod-uinf}
The forgetful functor is an isomorphism of categories between complete $\smash{\^{\#U}}$-modules and complete stackable $\#U$-meet-semilattices.
The $\le_r$ relations in a $\^{\#U}$-module are given by
\begin{equation*}
\rho(a, b) \le r  \iff  a \le_r b  \iff  a \le b \dotplus r.
\end{equation*}
\end{proposition}
\begin{proof}[Proof]
The characterization of $\le_r$ is proved as in \cref{thm:upos-met-r}.

Next, an Archimedean $\^{\#U}$-module $A$ is unstackable as a $\#U$-poset: by \cref{rmk:upos-stack}, it suffices to check that for $0 < r < 1$, $u := 1 \wedge (-)/r$, and $v := ((-) \dotminus r)/(1-r)$, if $u(a) \le u(b)$ and $v(a) \le v(b)$, then $a \le b$.
Let $s > 0$, and let $r(-) \le w \in \#U$ with equality on $[0, 1-s]$.
Then $1_\#I \le (w \circ u) \wedge (v^\times \circ v) \le (-) \dotplus rs$, whence from $u(a) \le u(b)$ and $v(a) \le v(b)$ we have $a \le b \dotplus rs$, i.e., $a \le_{rs} b$ by the above.
Since $A$ is Archimedean, it follows that $a \le b$.

If moreover $A$ is a complete $\^{\#U}$-module, then it is stackable: for $a, b \in A$ such that $v'(b) \le u'(a)$ for all $u', v' \in \#U$, with the same $s, u, v, w$ as above, letting $c_s := w(a) \wedge v^\times(b)$, we have $u(c_s) = u(w(a)) \wedge u(v^\times(b)) = u(w(a))$ which is within distance $s$ of $a$ since $1_\#U \le u \circ w \le (-) \dotplus s$, and $v(c_s) = v(w(a)) \wedge v(v^\times(b)) = v(v^\times(b)) = b$.
In particular, by unstackability (using \cref{thm:upos-unstack-r} and uniform continuity of $u$), the $c_s$ form a Cauchy net as $s \searrow 0$, hence converge to some $c$ such that $u(c) = a$ and $v(c) = b$.
Thus the forgetful functor restricts to the claimed subcategories.

The forgetful functor is full on Archimedean $\^{\#U}$-modules: the action by $w \in \^{\#U} \setminus \#U$ can be recovered from the $\#U$-action, since $w(a) = \top$ for $w(0) = 1$, while for $0 < w(0) < 1$, by unstackability, $w(a)$ is the unique element such that $u(w(a)) = \top$ and $v(w(a)) = (v \circ w)(a)$ where $u, v$ are as above for $r := w(0)$.
Thus $\#U$-equivariance implies $\^{\#U}$-equivariance.

Conversely, in a complete stackable $\#U$-meet-semilattice $A$, we may extend the $\#U$-action to a $\^{\#U}$-action by defining $w(a)$ for $0 < w(0) < 1$ to be the unique element as above.

The $\#U$-action on an Archimedean stackable $\#U$-poset $A$ preserves binary meets in $\#U$: for piecewise linear $u, v \in \#U$, we may show $(u \wedge v)(a) = u(a) \wedge v(a)$ by unstacking over a finite partition of $[0,1]$ on each piece of which $u, v$ are comparable; for arbitrary $u, v$, take piecewise linear approximations.

Finally, on a complete stackable $\#U$-meet-semilattice, the extended $\^{\#U}$-action from above also preserves binary meets in $\^{\#U}$, by a routine unstacking over $0 < w(0) < 1$.
\end{proof}

\begin{corollary}[of \cref{thm:ctslat-dual,thm:umod-uinf}]
\label{thm:ctslat-dual-umod}
Hom into $\#I$ yields a dual equivalence of categories between continuous lattices and complete $\^{\#U}$-modules.
\qed
\end{corollary}

%
%

We end by noting that we currently do not know whether complete $\^{\#U}$-modules can be equationally axiomatized, perhaps along the lines of \cite{Anachbin}, thereby showing that $\!{CtsLat}^\op$ is a variety.

\bigskip\noindent
Department of Mathematics \\
University of Michigan \\
Ann Arbor, MI 48109, USA \\
Email: \nolinkurl{ruiyuan@umich.edu}

\end{document}